\documentclass[sn-mathphys]{JORSC-jnl}

\jyear{2023}%

\theoremstyle{thmstyleone}%
\newtheorem{theorem}{Theorem}
\usepackage{amsthm} 
\newtheorem{example}{Example}%
\newtheorem{remark}{Remark}%

\theoremstyle{thmstylethree}%
\newtheorem{assumption}{Assumption}
\usepackage{indentfirst}
\usepackage{epstopdf}

\usepackage{booktabs}
\usepackage{pifont} 
\usepackage{adjustbox}
\usepackage{multirow} 

\newcommand{\cmark}{\ding{51}} 
\newcommand{\xmark}{\ding{55}} 

\usepackage{subcaption}
\usepackage{graphicx}
\usepackage{amsmath}
\usepackage{amssymb}
\usepackage{mathtools}

\def\R{{\mathbb R}}

\def\argmin{{\rm argmin}}
\begin{document}
\renewcommand{\qedsymbol}{}
\title[Lipschitz-free Projected Subgradient Method]{Lipschitz-free Projected Subgradient Method with Time-varying Step-size}


\author[1]{Yong Xia}\email{yxia@buaa.edu.cn}

\author[1]{Yanhao Zhang}\email{yanhaozhang@buaa.edu.cn}

\author*[1]{Zhihan Zhu}\email{zhihanzhu@buaa.edu.cn}

\affil[1]{\orgdiv{School of Mathematical Sciences}, \orgname{Beihang University}, \orgaddress{ \city{Beijing} \postcode{100191}, \country{People's Republic of China}}}


\abstract{We introduce a novel family of time-varying step-sizes for the classical projected subgradient method, offering optimal ergodic convergence. Importantly, this approach does not depend on the Lipschitz assumption of the objective function, thereby broadening the convergence result of projected subgradient method to non-Lipschitz case.}




\maketitle

\section{Introduction}
To tackle the nonsmooth convex optimization problem
\begin{equation*}
	x^*\in\argmin_{x\in\mathcal{X}}f(x),
\end{equation*}
where $\mathcal{X}\subset\R^n$ is a compact convex set enclosed within the Euclidean ball $B(x^*,R)$, and $f$ is a (possibly nonsmooth) convex function,  the traditional projected subgradient method (PSG) is employed as follows:
\begin{equation*}
	\left\{	\begin{aligned}
		y_{s+1}&=x_s-\eta_s g_s, ~\text{where } g_s\in\partial f(x_s),\\
		x_{s+1}&=  \argmin_{x\in\mathcal{X}}\Vert x-y_{s+1} \Vert,	\end{aligned} \right.
\end{equation*}
where $\|\cdot\|$ denotes the Euclidean norm throughout this paper.

In the literature, the following common Lipschitz assumption on $f$ is  made:
\begin{assumption}\label{ass}
	There exists an $L>0$ such that  for any $g\in\partial f(x)\neq\emptyset$ and $x\in \mathcal{X}$, it holds that $\Vert g\Vert\leqslant L$.
\end{assumption}
However, many commonly encountered convex functions fail to satisfy Assumption \ref{ass} on compact convex set.
\begin{example}\label{e1}
	The function $f(x)=-\sqrt{x}$ is a convex function on the interval $[0,1]$, which does not satisfy Assumption \ref{ass}.
\end{example}

It is well-established  that employing a constant step-size of
\[
\eta_s\equiv\frac{R}{L\sqrt{t}}, ~s=1,\cdots,t,\footnote{$t$ represents the current total number of iterations, while $s$ serves as the indicator of the iteration process.}
\]	
allows PSG to attain an optimal ergodic convergence rate, which is given by
\[
f\left(\frac{\sum_{s=1}^t x_s}{t}\right)-f(x^*)\leqslant \frac{RL}{\sqrt{t}},
\]
see, for example, \cite{nesterov2004lectures,nesterov2018lectures,bubeck2015convex}.

Recently, a time-varying step-size formula, as presented in \cite{nesterov2004lectures,nesterov2018lectures,bubeck2015convex}, given by
\begin{equation}
	\eta_s=\frac{R}{L\sqrt{s}},~s=1,\cdots,t,\label{size2}
\end{equation}
has been proven to ensure the optimal convergence rate of  PSG as well, as stated in \cite[Corollary 3.2]{Lan}. The following succinct result concerning PSG with the step-size given by \eqref{size2} is credited to \cite{Zhu}:
\begin{equation}
	f\left(\frac{\sum_{s=1}^t x_s}{t}\right)-f(x^*)\leqslant \frac{3RL}{2\sqrt{t}}. \label{conv}
\end{equation}

In \cite{nesterov2004lectures,nesterov2018lectures}, a more practical subgradient-normalized time-varying step-size is further examined, given by 
\begin{equation}
	\eta_s=\frac{R}{\|g_s\|\sqrt{s}},~s=1,\cdots,t, \label{Nesterov}
\end{equation}
which notably does not necessitate the knowledge of the Lipschitz constant beforehand. However, to guarantee the convergence of PSG, Assumption \ref{ass} is still required. Additionally, the convergence rate achieved is merely sub-optimal, given by
\begin{equation}\label{Nessub}
	f\left(\frac{\sum_{s=1}^t \eta_s x_s}{\sum_{s=1}^t \eta_s}\right)-f(x^*)\leqslant \frac{2RL+RL\log t}{4(\sqrt{t+1}-1)}.
\end{equation}
Although computing weighted average in the second half of the iterates \cite{Lan} or taking best iterate $\min_{s=1,\cdots,t}{f(x_s)}$ \cite{beck2017first} can achieve optimal rate, additional costs such as saving historical points or computing the function value at each iteration are required. Meanwhile, the results in \cite{Lan} and \cite{beck2017first} still say nothing on Example \ref{e1}.

For the sake of convenience, we refer to time-varying step-sizes \eqref{size2} and \eqref{Nesterov} as classic step-size and Nesterov step-size respectively in this paper.

The key contribution of this paper (see Section 2) is to introduce a family of subtle variations to Nesterov step-size \eqref{Nesterov}, which enables us to establish the optimal ergodic convergence rate of PSG, notably without requiring Assumption \ref{ass}. In Section 3, we numerically demonstrate the superiority of our step-size family over existing time-varying step-sizes. Conclusion is made in Section 4.

\section{A Family of Step-sizes}
We present the following result without the need for Assumption \ref{ass}. The proof is omitted here since a more generalized convergence analysis will be performed in Theorem \ref{Convergence1}.

\begin{theorem}\label{Convergence}
	For any fixed $a\in{[0,1]}$, PSG with the following step-size family   
		\begin{equation}
			\eta_s=\frac{R}{G_s s^{\frac{a}{2}}}, ~s=1,\cdots,t,\label{size3}
		\end{equation}
		where 
		\begin{equation}
			G_s = \max\{G_{s-1}, \|g_s\|s^{\frac{1-a}{2}}\} ~(G_0=-\infty), \label{G_s}
		\end{equation}
		satisfies
		\begin{equation}\label{conver}
			f\left(\frac{\sum_{s=1}^t x_s}{t}\right)-f(x^*)\leqslant
			\frac{3R}{2\sqrt{t}}\cdot \max_{s=1,\cdots,t}\|g_s\|.
		\end{equation}
	\end{theorem}

	\begin{remark}
		For Example \ref{e1}, convergence is still guaranteed as long as the initial point is not zero in our Lipschitz-free method, according to Theorem \ref{Convergence}.
	\end{remark}
	
	\begin{remark}
		In the scenario where $\partial f(x_s)$ is not a singleton, Theorem \ref{Convergence} suggests that selecting $g_s$ from $\partial f(x_s)$ with the minimal norm may possibly enhance the convergence. Once $\|g_s\| = 0$ for the selected subgradient,  it indicates that the global optimum has been reached and further iterations are no longer necessary.
	\end{remark}

	\begin{remark}\label{r1}
		Given Assumption \ref{ass}, Theorem \ref{Convergence} allows us to swiftly attain the optimal ergodic convergence result of \eqref{conv}.
	\end{remark}
	
	\begin{remark}\label{r2}
		Even when $\|g_s\|$ is unbounded (i.e., Assumption \ref{ass} is violated), convergence of PSG can still be assured by Theorem \ref{Convergence}, as long as the growth rate of $\|g_s\|$ during the iteration strictly stays  within $\mathcal{O}(\sqrt{s})$.
	\end{remark}
	
	\begin{remark}\label{r3}
		Theorem \ref{Convergence} also indicates that when $\max_{s=1,\cdots,t}\|g_s\|$ becomes large as iterations accumulate, restarting might be an effective strategy.
	\end{remark}
	
	
	The following theorem demonstrates that our Lipschitz-free method still retains the weak ergodic convergence result of PSG in \cite{Zhu}.
	
	\begin{theorem}\label{Convergence1}
		For any fixed $k\geqslant -1$, define $w_s^{(k)}$ \footnote{For the classic step-size \eqref{size2}, the weak ergodic weight \eqref{weight} here is equivalent to the weak ergodic weight in \cite{Zhu}.} as  
		\begin{equation}\label{weight}
			w_s^{(k)} = 
			\begin{cases} 
				1/\eta_s^k, & -1\leqslant k \leqslant 0, \\ 
				s^{\frac{k}{2}}, & ~~k > 0.
			\end{cases}
		\end{equation}
		PSG with step-size family \eqref{size3} satisfies   
		\begin{equation}\label{WeakConvergence}
			f\left( \frac{\sum_{s=1}^t w_s^{(k)} x_s}{\sum_{s=1}^t w_s^{(k)}} \right) - f(x^*)
			\leqslant \frac{t^{\frac{k+1}{2}} +\sum_{s=1}^{t} s^{\frac{k-1}{2}}}{2\sum_{s=1}^{t} s^{\frac{k}{2}}} R\max\limits_{s=1,\dots,t}{\|g_s\|}.
		\end{equation}
	\end{theorem}
	\begin{proof}
		Consider PSG with the step-size \eqref{size3}.   We have
		\begin{eqnarray}
			f(x_s)-f(x^*)&\leqslant& g_s^T(x_s-x^*) \nonumber~~~~({\rm by ~the~ definition~ of~  subgradient}) \\
			&=& \frac{1}{\eta_s}(x_s-y_{s+1})^T(x_s-x^*) \nonumber\\
			&=& \frac{1}{2\eta_s}(\Vert x_s-y_{s+1} \Vert^2+\Vert x_s-x^* \Vert^2-\Vert y_{s+1}-x^* \Vert^2)\label{eq1}\\
			&=& \frac{1}{2\eta_s}(\Vert x_s-x^* \Vert^2-\Vert y_{s+1}-x^* \Vert^2)+\frac{\eta_s}{2}\Vert g_s \Vert^2 \nonumber\\
			&\leqslant& \frac{1}{2\eta_s}(\Vert x_s-x^* \Vert^2-\Vert x_{s+1}-x^* \Vert^2)+\frac{\eta_s}{2}\Vert g_s \Vert^2,\label{ineq1}
		\end{eqnarray}
		where \eqref{eq1} is derived from the identity $2a^Tb=\Vert a\Vert^2+\Vert b\Vert^2-\Vert a-b\Vert^2$, and \eqref{ineq1} holds due to the fact that
		\begin{equation*}
			\Vert y_{s+1}-x^* \Vert^2\geqslant \Vert x_{s+1}-x^* \Vert^2,
		\end{equation*}
		which is a direct consequence of the projection theorem.
		
		By definition of $w_s^{(k)}$ in \eqref{weight}, we have 
		\begin{equation}\label{Monetone}
			\frac{w_s^{(k)}}{\eta_s}-\frac{w_{s-1}^{(k)}}{\eta_{s-1}}\geqslant0, ~s\geqslant2
		\end{equation}
		for any $k \geqslant -1$, since $1/\eta_s^{k+1}-1/\eta_{s-1}^{k+1}\geqslant0$ when $-1 \leqslant k \leqslant 0$ and $s^{\frac{k}{2}}/\eta_s - (s-1)^{\frac{k}{2}}/\eta_{s-1}\geqslant0$ when $k > 0$.
		
		By definition of $G_s$ in \eqref{G_s}, we also have the equivalent form of $G_s$ as follows:
		\begin{equation}\label{G_s1}
			G_s = \max\limits_{j=1,\dots,s}{\{\|g_j\|j^{\frac{1-a}{2}}\}}.
		\end{equation}
		
		Consequently, we have
		\begin{eqnarray}
			&& \left( \sum_{s=1}^t w_s^{(k)} \right) \left(f\left( \frac{\sum_{s=1}^t w_s^{(k)} x_s}{\sum_{s=1}^t w_s^{(k)}} \right) - f(x^*)\right)\nonumber \\
			&\leqslant&\sum_{s=1}^t w_s^{(k)}(f(x_s)-f(x^*))  ~~~~({\rm Jensen's~ inequality})\nonumber\\
			&\leqslant&\sum_{s=1}^t \frac{w_s^{(k)}}{2\eta_s }(\Vert x_s-x^* \Vert^2-\Vert x_{s+1}-x^* \Vert^2)+\sum_{s=1}^t\frac{w_s^{(k)}\eta_s}{2}\Vert g_s \Vert^2 ~~~~({\rm by~ \eqref{ineq1}}) \nonumber\\
			&=& \frac{w_1^{(k)}}{2\eta_1}\Vert x_1-x^* \Vert^2+\frac{1}{2}\sum_{s=2}^t\left(\frac{w_s^{(k)}}{ \eta_s }-\frac{w_{s-1}^{(k)}}{ \eta_{s-1}}\right)\Vert x_s-x^* \Vert^2\nonumber\\
			&& -\frac{w_t^{(k)}}{ 2\eta_{t} }\Vert x_{t+1}-x^* \Vert^2+\sum_{s=1}^t\frac{w_s^{(k)}\eta_s}{2}\Vert g_s \Vert^2\nonumber\\
			&\leqslant& R^2 \left( \frac{w_1^{(k)}}{2\eta_1}+\frac{1}{2}\sum_{s=2}^t\left(\frac{w_s^{(k)}}{ \eta_s }-\frac{w_{s-1}^{(k)}}{ \eta_{s-1}}\right) \right)+\sum_{s=1}^t\frac{w_s^{(k)}\eta_s}{2}\Vert g_s \Vert^2 ~~~~({\rm by~ \eqref{Monetone}})\nonumber\\
			&\leqslant& \frac{R^2 w_t^{(k)}}{2\eta_t}+\sum_{s=1}^t\frac{w_s^{(k)}\eta_s}{2}\Vert g_s \Vert^2. \label{sum}
		\end{eqnarray}
		
		For $-1 \leqslant k \leqslant 0$, we have
		\begin{eqnarray}
			&& f\left( \frac{\sum_{s=1}^t w_s^{(k)} x_s}{\sum_{s=1}^t w_s^{(k)}} \right) - f(x^*)\nonumber \\
			&\leqslant& \frac{\frac{R^2}{2\eta_t^{k+1}}+\sum_{s=1}^t\frac{\eta_s^{1-k}}{2}\Vert g_s \Vert^2}{\sum_{s=1}^t \eta_s^{-k}} ~~~~({\rm by~ \eqref{weight}~and~\eqref{sum}})\nonumber \\
			&\leqslant&
			\frac{R}{2}\frac{\left({G_t t^{\frac{a}{2}}}\right)^{1+k}+\sum_{s=1}^t\Vert g_s \Vert^2 \left({G_s s^{\frac{a}{2}}}\right)^{k-1}}{\sum_{s=1}^t \left({G_s s^{\frac{a}{2}}}\right)^{k}} ~~~~({\rm by~ \eqref{size3}})\nonumber \\
			&\leqslant&
			\frac{R}{2}\frac{\left({\max\limits_{s=1,\cdots,t}{\|g_s\|}}\right)^{1+k}t^{\frac{k+1}{2}}+\sum_{s=1}^t\Vert g_s \Vert^{1+k}s^{\frac{k-1}{2}}}{\left({\max\limits_{s=1,\cdots,t}{\|g_s\|}}\right)^{k}\sum_{s=1}^t s^{\frac{k}{2}}} \label{s1}\\
			&\leqslant&
			\frac{t^{\frac{k+1}{2}} +\sum_{s=1}^{t} s^{\frac{k-1}{2}}}{2\sum_{s=1}^{t} s^{\frac{k}{2}}} R\max\limits_{s=1,\dots,t}{\|g_s\|}, \nonumber
		\end{eqnarray}
		where \eqref{s1} holds by definition \eqref{G_s}, \eqref{G_s1} and $0 \leqslant a \leqslant 1$ when $-1 \leqslant k \leqslant 0$. And for $k > 0$, we have
		\begin{eqnarray}
			&& f\left( \frac{\sum_{s=1}^t w_s^{(k)} x_s}{\sum_{s=1}^t w_s^{(k)}} \right) - f(x^*)\nonumber \\
			&\leqslant& \frac{\frac{R^2 t^{\frac{k}{2}}}{2\eta_t}+\sum_{s=1}^t\frac{s^{\frac{k}{2}}\eta_s}{2}\Vert g_s \Vert^2}{\sum_{s=1}^t s^{\frac{k}{2}}} ~~~~({\rm by~ \eqref{weight}~and~\eqref{sum}})\nonumber \\
			&\leqslant& \frac{R}{2}\frac{t^{\frac{k}{2}}G_t t^{\frac{a}{2}}+\sum_{s=1}^t\frac{s^{\frac{k}{2}}\Vert g_s \Vert^2}{G_s s^{\frac{a}{2}}}}{\sum_{s=1}^t s^{\frac{k}{2}}} ~~~~({\rm by~ \eqref{size3}})\nonumber \\
			&\leqslant& \frac{R}{2}\frac{t^{\frac{k+1}{2}}\max\limits_{s=1,\cdots,t}{\|g_s\|}+\sum_{s=1}^t\Vert g_s \Vert s^{\frac{k-1}{2}}}{\sum_{s=1}^t s^{\frac{k}{2}}}\label{s2} \\
			&\leqslant&
			\frac{t^{\frac{k+1}{2}} +\sum_{s=1}^{t} s^{\frac{k-1}{2}}}{2\sum_{s=1}^{t} s^{\frac{k}{2}}} R\max\limits_{s=1,\dots,t}{\|g_s\|}, \nonumber
		\end{eqnarray}
		where \eqref{s2} follows from definition \eqref{G_s}, \eqref{G_s1} and $0 \leqslant a \leqslant 1$ when $k > 0$. The proof is complete.  \qed
	\end{proof}
	
	\begin{remark}\label{r5}
		By setting $k=-1$ in Theorem \ref{Convergence1}, we can immediately obtain the sub-optimal convergence rate \eqref{Nessub} in \cite{nesterov2004lectures,nesterov2018lectures,beck2017first}.
	\end{remark}
	
	\begin{remark}\label{r6}
		The optimal convergence rate \eqref{conver} in Theorem \ref{Convergence} serves as a special case by setting $k=0$ in Theorem \ref{Convergence1}.
	\end{remark}
	
	\begin{remark}\label{r7}
		By setting any $k$ such that $k>-1$ in Theorem \ref{Convergence1},  the convergence rate of  $f\left( \frac{\sum_{s=1}^t w_s^{(k)} x_s}{\sum_{s=1}^t w_s^{(k)}} \right)-f(x^*)$ will be $\mathcal{O}(1/\sqrt{t})$, which is optimal. For any fixed $t$, weak ergodic convergence will gradually lead to non-ergodic convergence as $k$ tends to infinity. This provides an asymptotic perspective on the non-ergodic convergence rate of the Lipschitz-free projected subgradient method.
	\end{remark}
	
	
	The theoretical properties of our step-size family \eqref{size3} compared to classic time-varying step-size \eqref{size2} and Nesterov step-size \eqref{Nesterov}, could be well illustrated in Table \ref{comparison}.
	
	\begin{table}[ht]
		\centering
		\caption{Comparison for convergence result between classic time-varying step-size \eqref{size2}, Nesterov step-size \eqref{Nesterov} and our step-size family \eqref{size3}}
		\begin{adjustbox}{width=1\textwidth}
			\begin{tabular}{@{}l@{\hspace{20pt}}c@{\hspace{20pt}}c@{\hspace{20pt}}c@{}}
				\toprule
				\multirow{2}{*}{\textbf{\raisebox{-1.6ex}{Convergence}}}   & \multicolumn{3}{c}{\textbf{Step-size}} \\ \cline{2-4}
				& \raisebox{-0.5ex}{\textbf{Classic} \eqref{size2}} & \raisebox{-0.5ex}{\textbf{Nesterov} \eqref{Nesterov}} & \raisebox{-0.5ex}{\textbf{Ours} \eqref{size3}} \\ \midrule
				Best Iterate           & \cmark           & \cmark                & \cmark           \\
				Ergodic ($k=-1$, Sub-optimal)              & \cmark           & \cmark           & \cmark            \\
				Ergodic ($k>-1$, Optimal)                 & \cmark           & \xmark           & \cmark           \\
				Without Lipschitz Constant                 & \xmark           & \cmark           & \cmark     \\
				Without Lipschitz Assumption                   & \xmark           & \xmark           & \cmark     \\ \bottomrule
			\end{tabular}
		\end{adjustbox}
		\label{comparison}
	\end{table}
	
	\section{Experiment}
	We consider the ball-constrained Lasso problem formulated as follows:
	\begin{equation}
		\min_{x\in B(R)} \|y- \Phi x\|_2^2 + \lambda \|x\|_1,
	\end{equation}
	where $x\in \mathbb{R}^N$, $\Phi \in \mathbb{R}^{M \times N}$ and $y \in \mathbb{R}^M$. Specifically, we set $N = 512$, $M = 300$, $R=50$ and the regularization parameter $\lambda=10$.
	In this test, $\Phi$ is a random Gaussian matrix, hence the Lipschitz constant for objective function is difficult to be estimated in advance.
	
	\subsection{Sensitivity to $a$ in Step-size Family \eqref{size3}}
	The effect of $a$ on the curve of the best function value (i.e., function value of best iterate) are compared across different test cases as follows.
	
	\begin{figure}[h]
		\centering
		\begin{subfigure}[b]{0.48\linewidth}
			\centering
			\includegraphics[width=0.99\linewidth]{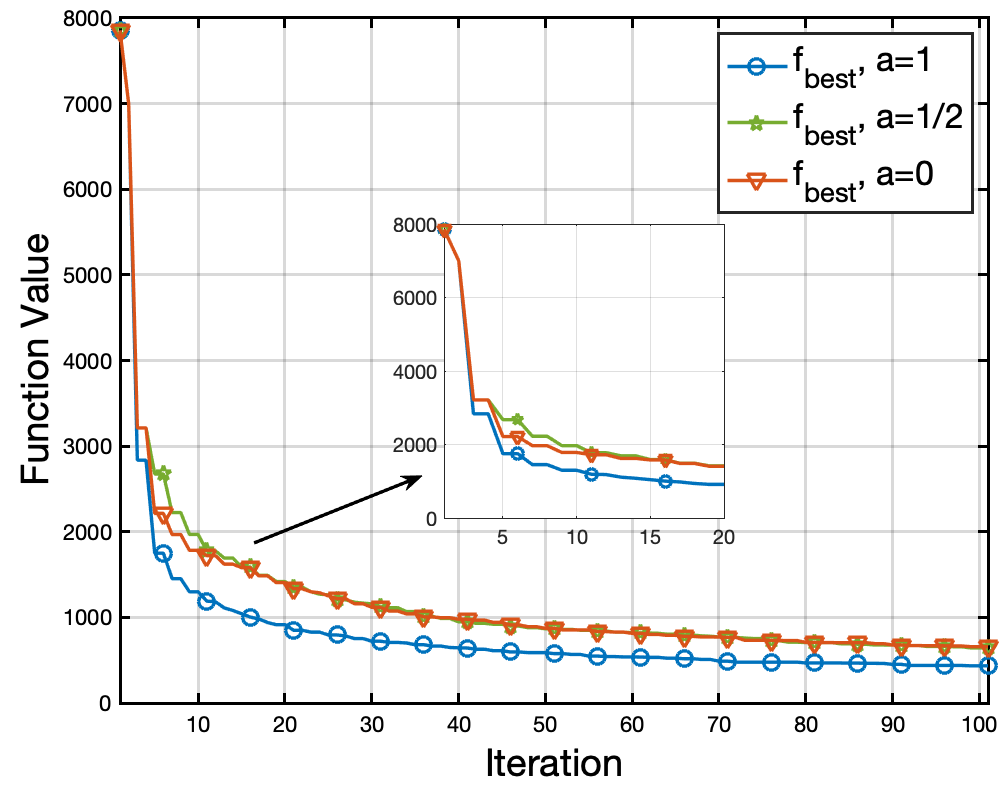}
		\end{subfigure}
		\begin{subfigure}[b]{0.48\linewidth}
			\centering
			\includegraphics[width=0.99\linewidth]{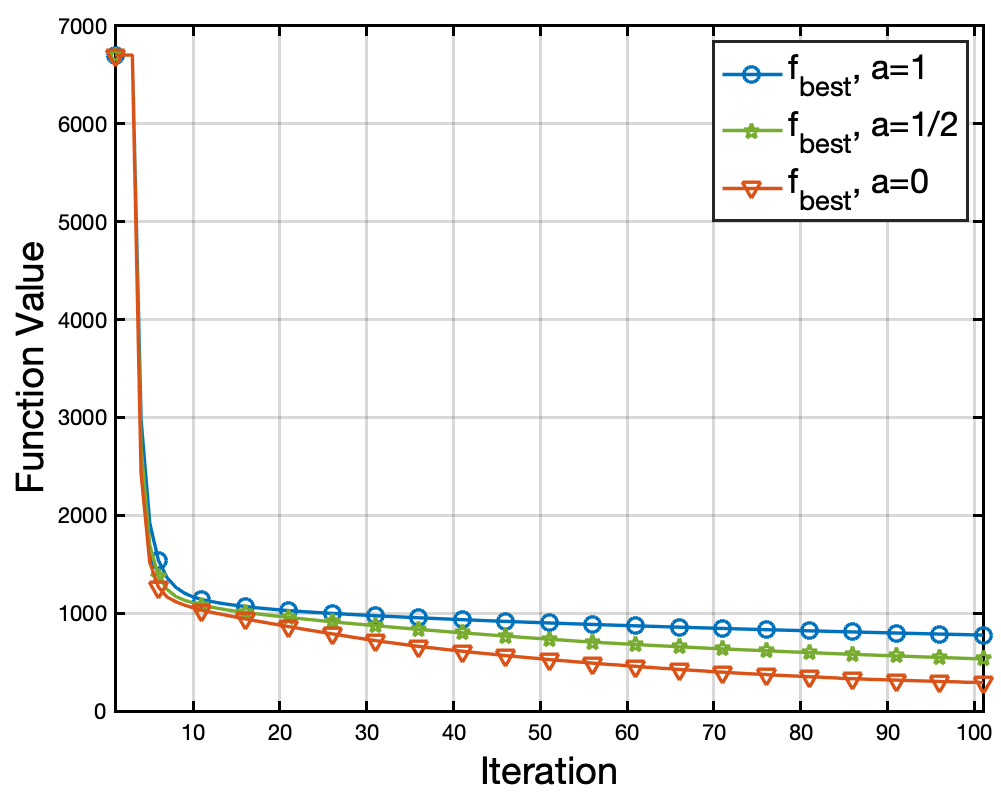}
		\end{subfigure}
		\caption{Sensitivity to value of $a$ on different testing scenarios}
		\label{initial pic}
	\end{figure}
	
	From the experimental results, convergence by our step-size family \eqref{size3} is always guaranteed regardless of the value of $a$. However, different values of $a$ indeed have an impact on the final result. In practice, $a$ could be adjusted according to the specific problem, and an appropriate choice of $a$ may lead to further improvements.
	
	\subsection{Comparison with Step-size \eqref{Nesterov} and Benefits of Weak Ergodicity}
	We further compare the performance of our step-size family \eqref{size3} and Nesterov step-size \eqref{Nesterov}, and demonstrate the practical benefits of weak ergodicity. In this experiment, the value of  $a$ is fixed as 1.
	
	\begin{figure}[h]
		\centering
		\begin{subfigure}[b]{0.48\linewidth}
			\centering
			\includegraphics[width=0.99\linewidth]{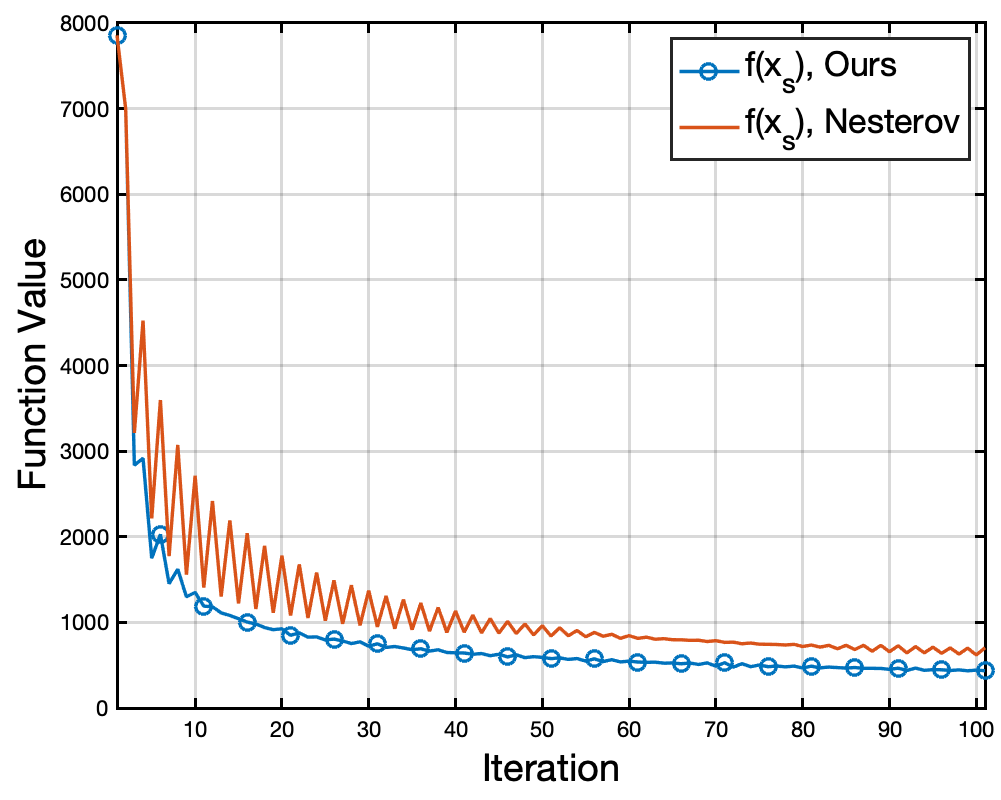}
			\caption{Function values by our step-size family \eqref{size3} and Nesterov step-size \eqref{Nesterov}}
			\label{weak_ergodic1}
		\end{subfigure}
		\begin{subfigure}[b]{0.48\linewidth}
			\centering
			\includegraphics[width=0.99\linewidth]{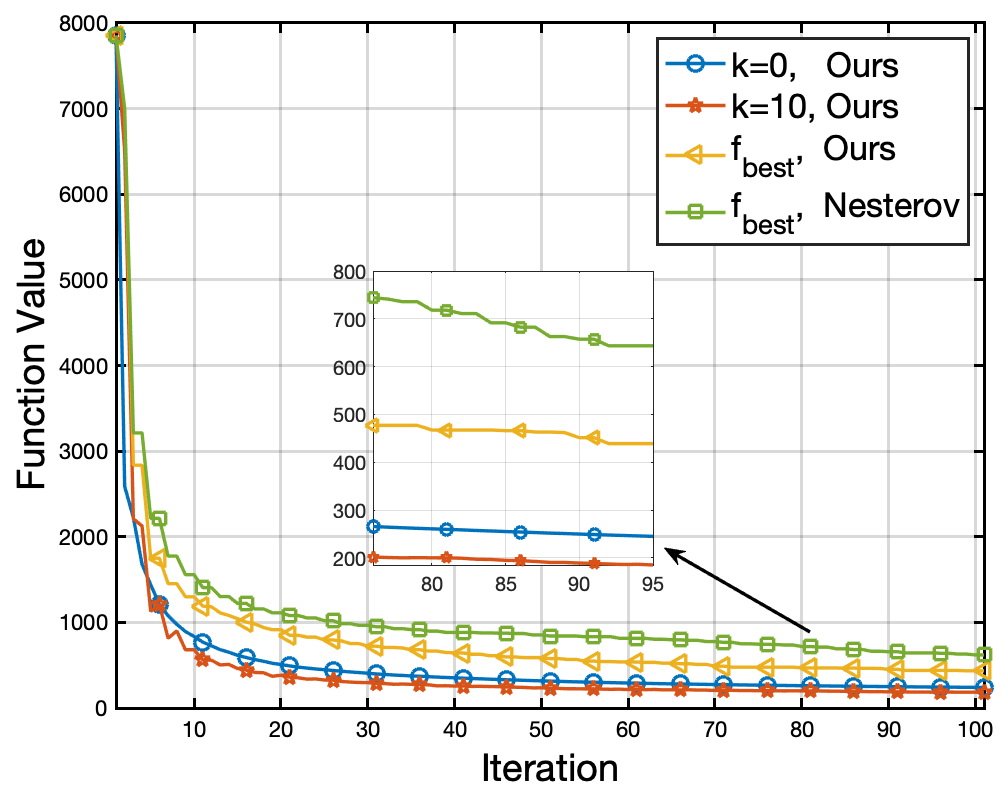}
			\caption{Comparison for best iterate and weak ergodic convergence}
			\label{weak_ergodic2}
		\end{subfigure}
		\caption{Comparison with step-size \eqref{Nesterov} and benefits of weak ergodic convergence}
		\label{weak_ergodic}
	\end{figure}
	
	As shown in Fig. \ref{weak_ergodic1}, the function values generated by Nesterov step-size \eqref{Nesterov} exhibit large oscillations, whereas the function values produced by our step-size family \eqref{size3} are notably more stable. In Fig. \ref{weak_ergodic2}, although all four strategies  achieve the optimal convergence rate theoretically, our step-size family \eqref{size3} shows evidently better performance in the experiments compared to Nesterov step-size \eqref{Nesterov}. Additionally, the experimental results reveal that weak ergodic convergence \eqref{WeakConvergence} not only eliminates the need to compute the function values in each iteration, but also provides additional gains in practice by assigning greater weight to the most recent points, which is also consistent with intuition.
	
	\section*{Conclusion}
	In this paper, we proposed a family of time-varying step-sizes, based on which Lipschitz-free projected subgradient method was established. Without relying on the Lipschitz assumption, we derived the optimal ergodic convergence rate and weak ergodic convergence theory for Lipschitz-free methods, thereby extending the convergence results of projected subgradient method to the non-Lipschitz case. Our step-size family not only possesses the best theoretical properties but also shows evident improvements in experiments compared to traditional step-sizes. Future work includes extending the Lipschitz-free results to mirror descent and other schemes with time-varying step sizes for solving nonsmooth convex optimization.
	

%
%
%
%
%
%
%
%
\bmhead{Funding}

This work was supported by National Natural Science Foundation of China under grant 12171021, and the Academic Excellence Foundation of BUAA for PhD Students.

\bmhead{Data Availability}
The manuscript has no associated data.

%

\bmhead{Conflict of interest}

The authors declare that there is no conflict of interest.
\end{document}